\def\Diff{\operatorname{Diff}}
\def\diff{\operatorname{diff}}
\def\Homeo{\operatorname{Homeo}}
\def\Sp{\operatorname{Sp}}

\documentclass[10pt,reqno]{amsart}
\usepackage{amssymb}
     \makeatletter
     \def\section{\@startsection{section}{1}%
     \z@{.7\linespacing\@plus\linespacing}{.5\linespacing}%
     {\bfseries%\normalfont\scshape
     \centering
     }}
     \def\@secnumfont{\bfseries}
     \makeatother
\newtheorem{theorem}{Theorem}[section]

\newtheorem{proposition}[theorem]{Proposition}

\theoremstyle{definition}
\newtheorem{definition}[theorem]{Definition}
\theoremstyle{definition}
\newtheorem{notation}[theorem]{Notation}

\theoremstyle{remark}
\newtheorem{remark}[theorem]{Remark}
\newtheorem*{acknowledgement}{Acknowledgement}
\numberwithin{equation}{section} 
\begin{document}

\title[A Brownian motion on $\Diff(S^1)$]
{A Brownian motion on $\Diff(S^1)$}

%\author{Maria Gordina}
%\address{Department of Mathematics, University of Connecticut, Storrs, CT 06269, U.S.A. }
%\email{gordina@math.uconn.edu}

\author{Mang Wu{$^{\ast}$}}
%\author{Mang Wu}{$^{\ast}$}
\thanks{{$^{\ast}$} This research was partially supported by NSF Grant
DMS-0706784.} \email{mwu@math.uconn.edu}



\begin{abstract}
Let $\Diff(S^1)$ be the group of orientation preserving $C^\infty$
diffeomorphisms of $S^1$. In \cite{Malliavin1999} P. Malliavin and
then in \cite{Fang2002} S. Fang constructed a canonical Brownian
motion associated with the $H^{3/2}$ metric on the Lie algebra
$\diff(S^1)$. The canonical Brownian motion they constructed lives
in the group $\Homeo\left( S^{1}\right)$ of H\"olderian
homeomorphisms of $S^1$, which is larger than the group
$\Diff(S^1)$. In this paper, we present another way to construct a
Brownian motion that lives in the group $\Diff(S^1)$, rather than in
the larger group $\Homeo\left( S^{1}\right)$.
\end{abstract}

\maketitle

\section{Introduction}\label{section1}

Let $\Diff(S^1)$ be the group of orientation preserving
$C^\infty$-diffeomorphisms of $S^1$, and let $\diff(S^1)$ be the
space of $C^\infty$-vector fields on $S^1$. The space $\diff(S^1)$
can be identified with the space of $C^\infty$-functions on $S^1$.
Therefore, $\diff(S^1)$ carries a natural Fr\'echet space structure.
In addition $\diff(S^1)$ is an infinite dimensional Lie algebra: for
any $f,g\in\diff(S^1)$, the Lie bracket is given by $[f,g]=f'g-fg'$.
Thus the group $\Diff(S^1)$ associated with the Lie algebra
$\diff(S^1)$ becomes an infinite dimensional Fr\'echet Lie group
\cite{Milnor}. Our goal in this paper is to construct a Brownian
motion in the group $\Diff(S^1)$.

In general, to construct a Brownian motion in a Lie group, one might
solve a Stratonovich stochastic differential equation (SDE) on such
a group. The method is best illustrated for a finite dimensional
compact Lie group.

Let $G$ be a finite dimensional compact Lie group. Denote by
$\mathfrak{g}$ the Lie algebra of $G$ identified with the tangent
space $T_eG$ to the group $G$ at the identity element $e\in G$. Let
$L_g:G\to G$ be the left translation of $G$ by an element $g\in G$,
and let $(L_g)_\ast:\mathfrak{g}\to T_gG$ be the differential of
$L_g$. If we choose a metric on $\mathfrak{g}$ and let $W_t$ be the
standard Brownian motion on $\mathfrak{g}$ corresponding to this
metric, we can develop the Brownian motion $W_t$ onto $G$ by solving
a Stratonovich stochastic differential equation
\begin{equation}\label{eq1.1}
\delta \widetilde{X}_t = (L_{\widetilde{X}_t})_\ast\delta W_t
\end{equation}
where $\delta$ stands for the Stratonovich differential. The
solution $\widetilde{X}_t$ is a Markov process on $G$ with the
generator being the Laplace operator on $G$. We call
$\widetilde{X}_t$ the Brownian motion on the group $G$ \cite{Hsu,
Kunita}.

In case when $G$ is an infinite dimensional Hilbert Lie group, one
can solve Equation \eqref{eq1.1} by using the theory of stochastic
differential equations in Hilbert spaces as developed by G. DaPrato
and J. Zabczyk in \cite{DaPrato}. Using this method, M. Gordina
\cite{Gordina1, Gordina2, Gordina3} and M. Wu \cite{Wu} constructed
a Brownian motion in several Hilbert-Schmidt groups. The
construction relied on the fact that these Hilbert-Schmidt groups
are Hilbert Lie groups.

In the present case, we would like to replace $G$ by $\Diff(S^1)$
and $\mathfrak{g}$ by $\diff(S^1)$ and solve Equation \eqref{eq1.1}
correspondingly. But because the group $\Diff(S^1)$ is a Fr\'echet
Lie group, which is not a Hilbert Lie group, Equation \eqref{eq1.1}
does not even make sense as it stands. First, we need to interpret
the Brownian motion $W_t$ in the Fr\'echet space $\diff(S^1)$
appropriately. Second, we are lacking a well developed stochastic
differential equation theory in Fr\'echet spaces to make sense of
Equation \eqref{eq1.1}.

In 1999, P. Malliavin \cite{Malliavin1999} first constructed a
canonical Brownian motion on $\Homeo(S^1)$, the group of H\"olderian
homeomorphisms of $S^1$. In 2002, S. Fang \cite{Fang2002} gave a
detailed construction of this canonical Brownian motion on the group
$\Homeo(S^1)$. Their constructions were essentially by interpreting
and solving the same Equation \eqref{eq1.1} on the group
$\Diff(S^1)$.

To define the Brownian motion $W_t$ in Equation \eqref{eq1.1},
P.~Malliavin and S.~Fang chose the $H^{3/2}$ metric of the Lie
algebra $\diff(S^1)$. Basically, this metric uses the set
\begin{equation}\label{eq1.2}
\{n^{-3/2}\cos(n\theta), m^{-3/2}\sin(m\theta)| m,n=1,2,3,\cdots\},
\end{equation}
which is a subset of the Lie algebra $\diff(S^1)$, as an orthonormal
basis to form a Hilbert space $H^{3/2}$. Then they defined $W_t$ to
be the cylindrical Brownian motion in $H^{3/2}$ with the covariance
operator being the identity operator on $H^{3/2}$. But since the
coefficients $n^{-3/2}$ and $m^{-3/2}$ do not decrease rapidly
enough, the Hilbert space $H^{3/2}$ is not contained in the Lie
algebra $\diff(S^1)$. Therefore, the Brownian motion $W_t$ they
defined on $H^{3/2}$ does not live in $\diff(S^1)$ either. This is
the essential reason why the canonical Brownian motion they
constructed lives in a larger group $\Homeo(S^1)$, but not in the
group $\Diff(S^1)$.

To interpret and solve Equation \eqref{eq1.1}, S.~Fang treated it as
a family of stochastic differential equations on $S^1$: for each
$\theta\in S^1$, S.~Fang considered the equation
\begin{equation}\label{eq1.3}
\delta \widetilde{X}_{\theta,t}
=(L_{\widetilde{X}_{\theta,t}})_\ast \delta W_{\theta,t},
\end{equation}
which is a stochastic differential equation on $S^1$. By solving
Equation \eqref{eq1.3} for each $\theta\in S^1$,S.~Fang obtained a
family of solutions $\widetilde{X}_{\theta,t}$ parameterized by
$\theta$. Then he used a Kolmogorov type argument to show that the
family $\widetilde{X}_{\theta,t}$ is H\"olderian continuous in the
variable $\theta$. Using this method, he proved that for each $t\ge
0$, $\widetilde{X}_{\theta,t}$ is a H\"olderian homeomorphism of
$S^1$. Thus, he constructed the canonical Brownian motion on the
group $\Homeo(S^1)$. But this Kolmogorov type argument cannot be
pushed further to show that $\widetilde{X}_{\theta,t}$ is
differentiable in $\theta$. Therefore,S.~Fang's method does not seem
to be suitable to construct a Brownian motion that lives in the
group $\Diff(S^1)$, rather than in $\Homeo\left( S^{1}\right)$.

In the current paper, our goal is to construct a Brownian motion
that lives in the group $\Diff(S^1)$. To achieve this, we need
another way to interpret and solve Equation \eqref{eq1.1}.

First, instead of the $H^{3/2}$ metric that P.~Malliavin and S.~Fang
used, we choose a very ``strong'' metric on the Lie algebra
$\diff(S^1)$: let $\{\lambda(n)\}_{n=1}^\infty$ be a sequence of
rapidly decreasing positive numbers. We use the set
\begin{equation}\label{eq1.4}
\{\lambda(n)\cos(n\theta),\lambda(m)\sin(m\theta) | m,n=1,2,3,\cdots \},
\end{equation}
which is a subset of the Lie algebra $\diff(S^1)$, as an orthonormal
basis to form a Hilbert space $H_\lambda$. Then we define the
Brownian motion $W_t$ to be the cylindrical Brownian motion in
$H_\lambda$ with the covariance operator being the identity operator
on $H_\lambda$. Because the coefficients $\lambda(n)$ are rapidly
decreasing, the Hilbert space $H_\lambda$ is a \emph{subspace} of
the Lie algebra $\diff(S^1)$. Therefore, the Brownian motion $W_t$
lives in the Lie algebra $\diff(S^1)$, and the solution of Equation
\eqref{eq1.1} has a better chance to live in the group $\Diff(S^1)$.

Second, in contrast to Fang's method of interpreting Equation
\eqref{eq1.1} ``pointwise'' as a family of stochastic differential
equations on $S^1$, we interpreted it as a sequence of stochastic
differential equations on a sequence of ``Hilbert'' spaces. To do
this, we embed the group $\Diff(S^1)$ into an affine space
$\widetilde{\diff}(S^1)$ that is isomorphic to the Lie algebra
$\diff(S^1)$. Let $H^k$ be the $k$th Sobolev space over $S^1$. It is
a separable Hilbert space. Let $\widetilde{H}^k$ be the
corresponding affine space that is isomorphic to $H^k$. For the
precise definition of the space $\widetilde{\diff}(S^1)$ and
$\widetilde{H}^k$, see Section 2. It is well known that the space
$\diff(S^1)$ is the intersection of the Sobolev spaces $H^k$.
Similarly, $\widetilde{\diff}(S^1)$ is the intersection of the
affine spaces $\widetilde{H}^k$. Now we have the embedding
\begin{equation}\label{eq1.5}
\Diff(S^1)\subseteq \widetilde{\diff}(S^1) \subseteq \widetilde{H}^k,
\hspace{.2in}
k=1,2,3,\cdots
\end{equation}
Thus, we can interpret Equation \eqref{eq1.1} as a sequence of
stochastic differential equations on the sequence of affine spaces
$\{\widetilde{H}^k\}_{k=1}^\infty$ each of which is isomorphic to
the Hilbert space $H^k$. These stochastic differential equations can
be solved by DaPrato and Zabczyk's method \cite{DaPrato}.

In accordance with the notations used by DaPrato and Zabczyk in
\cite{DaPrato}, in the rest of this paper, we will denote the
operator $(L_{\widetilde{X}_t})_\ast$ in Equation \eqref{eq1.1} by
$\widetilde{\Phi}(\widetilde{X}_t)$. The operator $\widetilde{\Phi}$
will be discussed in detail in Section 2. After adding the initial
condition, we can now re-write Equation \eqref{eq1.1} as
\begin{equation}\label{eq1.6}
\delta \widetilde{X}_t=\widetilde{\Phi}(\widetilde{X}_t)\delta W_t,
\hspace{.2in}
\widetilde{X}_0=id
\end{equation}
where $id$ is the identity element in $\Diff(S^1)$.

Equation \eqref{eq1.6} is interpreted as a stochastic differential
equation in each ``Hilbert'' space $\widetilde{H}^k$. To use DaPrato
and Zabczyk's method to solve this equation, we need to establish
the Lipschitz condition of the operator $\widetilde{\Phi}$. In
Section 2, it turns out that the operator $\widetilde{\Phi}$ is
\emph{locally} Lipschitz. So the explosion time of the solution
needs to be discussed.

After solving Equation \eqref{eq1.6} in $\widetilde{H}^k$ for each $k$,
it is relatively easy to prove that the solution
lives in the affine space $\widetilde{\diff}(S^1)$ (Proposition \ref{intersection}).
By the embedding \eqref{eq1.5}, the group $\Diff(S^1)$ is a subset
of the affine space $\widetilde{\diff}(S^1)$.
We wish to push one step further to prove that the solution actually
lives in the group $\Diff(S^1)$.

In general, to prove a process lives in a group rather than in an ambient space,
one needs to construct an inverse process.
To construct the inverse process, usually one needs to solve another
stochastic differential equation -- the SDE for the inverse process \cite{Gordina1, Wu}.
In our case, we have derived the SDE for the inverse process:
\begin{equation}\label{eq1.7}
\delta\widetilde{Y}_t=\widetilde{\Psi}(\widetilde{Y}_t)\delta W_t
\end{equation}
where $\widetilde{\Psi}$ is an operator such that for $\tilde{g}\in\Diff(S^1)$
and $f\in\diff(S^1)$, $\widetilde{\Psi}(\tilde{g})f=D\tilde{g}\cdot f$,
where $D=d/d\theta$ and ``$\cdot$'' is the pointwise multiplication of two functions.
Because the operator $D$ causes loss of one degree of smoothness,
we cannot interpret Equation \eqref{eq1.7} in $\widetilde{H}^k$ as we did for
Equation \eqref{eq1.6}, and we were forced to give up this method.

But we managed to get around this problem.
We first observed that
an element $\tilde{f}\in\widetilde{\diff}(S^1)$ belongs to $\Diff(S^1)$
if and only if $\tilde{f}'(\theta)>0$ for all $\theta\in S^1$.
Based on this observation, we showed that the solution is contained in
the group $\Diff(S^1)$ up to a stopping time.
Then we can ``concatenate'' this small piece of solution with another
small piece of solution to make a new solution up to a longer stopping time.
The key idea is Proposition (\ref{Sasha}) and the following remark (Remark \ref{concatenate}).
Finally, we were able to prove the following theorem (Theorem \ref{MainTheorem}):
\begin{theorem}\label{theorem1.1}
There is a unique $\widetilde{H}^k$-valued solution with continuous sample paths
to Equation \eqref{eq1.6} for all $k=0,1,2,\cdots$.
Furthermore, the solution is non-explosive and lives in the group $\Diff(S^1)$.
\end{theorem}

\begin{acknowledgement}
The author would like to thank Maria Gordina for her advising through out
the preparation of the paper.
The author would like to thank Alexander Teplyaev for pointing out
the important idea which helped us to prove Theorem
\ref{theorem1.1}.
The author also would like to thank Matt Cecil
for many helpful discussions.
\end{acknowledgement}

\section{An interpretation of Equation \eqref{eq1.6}}\label{section2}

\subsection{The group $\Diff(S^1)$ and the Lie algebra $\diff(S^1)$}

Let $\Diff(S^1)$ be the group of orientation preserving $C^\infty$
diffeomorphisms of $S^1$, and $\diff(S^1)$ be the space of
$C^\infty$ vector fields on $S^1$. We have the following
identifications for the space $\diff(S^1)$:
\begin{align}
\diff(S^1) &\cong \{ f:S^1\to\mathbb{R} : f\in C^\infty \}\label{e.2.1} \\
&\cong \{ f:\mathbb{R}\to\mathbb{R} : f\in C^\infty, f(x)=f(x+2\pi),
\textrm{ for all }x\in\mathbb{R}\} \notag
\end{align}
Using this identification, we see that the space $\diff(S^1)$ has a
Fr\'{e}chet space structure. In addition, this space  has a Lie
algebra structure, namely, for $f,g\in\diff(S^1)$ the Lie bracket is
given by
\begin{equation}
[f,g]=f'g-fg',
\end{equation}
where $f'$ and $g'$ are derivatives with respect to the variable
$\theta\in S^1$. Therefore, the group $\Diff(S^1)$ is a Fr\'{e}chet
Lie group as defined in \cite{Milnor}.
\begin{notation}
Using the above identification \ref{e.2.1}, we also have an
identification for $\Diff(S^1)$
\begin{equation}\label{eq2.3}
\Diff(S^1) \cong \{ \tilde{f}:\mathbb{R}\to\mathbb{R} :
\tilde{f}=id+f, f\in\diff(S^1), \tilde{f}'>0 \},
\end{equation}
where $id$ is the identity function from $\mathbb{R}$ to
$\mathbb{R}$. We note that the set on the right hand side of the
above identification is a group with the group multiplication being
composition of functions. We require that for
$\tilde{f},\tilde{g}\in\Diff(S^1)$,
$\tilde{f}\tilde{g}=\tilde{g}\circ\tilde{f}$. Under this
identification, the left translation of $\Diff(S^1)$ is given by
$L_{\tilde{g}}\tilde{f}=\tilde{g}\tilde{f}=\tilde{f}\circ\tilde{g}$.

Denote
\begin{equation}\label{eq2.4}
\widetilde{\diff}(S^1)= \{ \tilde{f}:\mathbb{R}\to\mathbb{R} |
\tilde{f}=id+f, f\in\diff(S^1) \}
\end{equation}
\end{notation}
The space $\widetilde{\diff}(S^1)$ is an affine space
which is isomorphic to the vector space $\diff(S^1)$. We denote the
isomorphism by $\sim$, that is,
$\sim:\diff(S^1)\to\widetilde{\diff}(S^1)$,
$f\mapsto\tilde{f}=id+f$. Comparing \eqref{eq2.3} and \eqref{eq2.4},
we have the embedding
\begin{equation}\label{eq2.5}
\Diff(S^1)\subseteq\widetilde{\diff}(S^1).
\end{equation}
With this embedding, the differential of a left translation
$L_{\tilde{g}}$ becomes
$(L_{\tilde{g}})_{\ast}:\diff(S^1)\to\diff(S^1)$, and is given by
$(L_{\tilde{g}})_\ast f= f\circ\tilde{g}$ for $f\in\diff(S^1)$.

The following proposition is an immediate observation from the
identification \eqref{eq2.3} and  definition of
$\widetilde{\diff}(S^1)$ given by \eqref{eq2.4}. Yet, it plays a key
role in proving the main theorem Theorem \ref{theorem1.1}.
\begin{proposition}\label{Difference}
An element $\tilde{f}\in\widetilde{\diff}(S^1)$ belongs to $\Diff(S^1)$
if and only if $\tilde{f}'>0$, or equivalently $f'>-1$.
\end{proposition}

\vspace{1pt}
\subsection{The Hilbert space $H_\lambda$ and the Brownian motion $W_t$}

To define the Brownian motion $W_t$ in Equation \eqref{eq1.6}, we
need to choose a metric on the Lie algebra $\diff(S^1)$. Comparing
with the $H^{3/2}$ metric that P.~Malliavin and S.~Fang chose, the
metric we choose here is a very ``strong'' metric.

\begin{definition}\label{scaling}
Let $\mathcal{S}$ be the set of even functions $\lambda:\mathbb{Z}\to (0, \infty)$
 such that $\lim_{n\to\infty} |n|^k\lambda(n)=0$ for all $k\in\mathbb{N}$.
For $\lambda\in\mathcal{S}$,
let $\hat{e}_n=\hat{e}_n^{(\lambda)}\in\diff(S^1)$ be defined by
\begin{equation}
\hat{e}_n^{(\lambda)}(\theta)=\left\{
\begin{array}{ll}
\lambda(n)\cos(n\theta), & n\ge0 \\
\lambda(n)\sin(n\theta), & n<0
\end{array}
\right.
\end{equation}
Let $H_{\lambda}$ be the
Hilbert space with the set $\{\hat{e}_n^{(\lambda)}\}_{n\in\mathbb{Z}}$ as an
orthonormal basis.
\end{definition}
Note that  the function $\lambda$ is rapidly decreasing, therefore
the Hilbert space $H_\lambda$ defined above is a \emph{proper
subspace} of $\diff(S^1)$. We also remark that
$\diff(S^1)=\bigcup_{\lambda\in\mathcal{S}} H_\lambda$.

Let $\alpha,\lambda\in\mathcal{S}$ be defined by
$\lambda(n)=|n|\alpha(n)$, and let $H_\alpha$ and $H_\lambda$ be the
corresponding Hilbert subspaces of $\diff(S^1)$. Then we have
$H_\alpha\subset H_\lambda$, and the inclusion map
$\iota:H_\alpha\hookrightarrow H_\lambda$ that sends
$\hat{e}_n^{(\alpha)}$ to
$\hat{e}_n^{(\alpha)}=\frac{1}{|n|}\hat{e}_n^{(\lambda)}$ is a
Hilbert-Schmidt operator. The adjoint operator
$\iota^\ast:H_\lambda\to H_\alpha$ that sends
$\hat{e}_n^{(\lambda)}$ to $\frac{1}{|n|}\hat{e}_n^{(\alpha)}$ is
also a Hilbert-Schmidt operator. The operator
$Q_\lambda=\iota\iota^\ast:H_\lambda\to H_\lambda$ is a trace class
operator on $H_\lambda$, and $H_\alpha=Q_\lambda^{1/2} H_\lambda$.

\begin{definition}\label{Wt}
Let $W_t$ be a Brownian motion defined by
\begin{equation}
W_t=\sum_{n\in\mathbb{Z}} B_t^{(n)}\hat{e}_n^{(\alpha)}
=\sum_{n\in\mathbb{Z}} \frac{1}{|n|}B_t^{(n)}\hat{e}_n^{(\lambda)}
\end{equation}
where $\{B_t^{(n)}\}_n$ are mutually independent standard $\mathbb{R}$-valued Brownian motions.
\end{definition}

We see that $W_t$ is a cylindrical Brownian motion on $H_\alpha$
with the covariance operator being the identity operator on
$H_\alpha$. Also, $W_t$ is a Brownian motion on $H_\lambda$ with the
covariance operator being the operator $Q_\lambda$.

\vspace{1pt}
\subsection{The Sobolev space $H^k$ and the affine space $\widetilde{H}^k$}

Now we turn to the Sobolev spaces over $S^1$. Let us first recall
some basic properties of the Sobolev spaces over $S^1$ found for example in
\cite{Adams}.

Let $k$ be a non-negative integer. Denote by  $C^k$ the space of
$k$-times continuously differentiable real-valued functions on
$S^1$, and denote by $H^k$ the $k$th Sobolev space on $S^1$. Recall
that $H^k$ consists of functions $f: S^1 \to \mathbb{R}$ such that
$f^{(k)}\in L^2$, where $f^{(k)}$ is the $k$th derivative of $f$ in
distributional sense. The Sobolev space $H^k$ has a norm given by
\begin{equation}
\|f\|_{H^k}^2=\|f\|_{L^2}^2 + \|f^{(k)}\|_{L^2}^2
\end{equation}
The Sobolev space $H^k$ is a separable Hilbert space, and $C^k$ is a
dense subspace of $H^k$. We will make use of the following standard
properties of the spaces $H^{k}$.

\begin{theorem}[\cite{Adams}]\label{SobolevFacts}
Let $m,k$ be two non-negative integers.
\begin{enumerate}
\item\label{Sobolev1}
If $m\le k$ and $f\in H^k$, then $\|f\|_{H^m}\le\|f\|_{H^k}$.
\item\label{Sobolev2}
If $m<k$ and $f\in H^k$, then there exists a constant $c_k$ such that
$ \|f^{(m)}\|_{L^\infty}\le c_k \|f\|_{H^k} $.
\item\label{Sobolev3}
$H^{k+1}\subseteq H^k$ for all $k=0,1,2,\cdots$, and $\diff(S^1)=\bigcap_{k=0}^\infty H^k$.
\end{enumerate}
\end{theorem}

An element $f\in H^k$ can be identified with a $2\pi$-periodic
function from $\mathbb{R}$ to $\mathbb{R}$. Define
\begin{equation}
\widetilde{H}^k=\{\tilde{f}:\mathbb{R}\to\mathbb{R}: \tilde{f}=id+f,
f\in H^k \}
\end{equation}
Then $\widetilde{H}^k$ is an affine space that is isomorphic to the
Sobolev space $H^k$. We denote the isomorphism by $\sim$, that is,
$\sim:H^k\to\widetilde{H}^k$, $f\mapsto\tilde{f}=id+f$. The image of
$C^k$ under the isomorphism, denoted by $\widetilde{C}^k$, is a
dense subspace of the affine space $\widetilde{H}^k$. An element
$\tilde{f}\in\widetilde{H}^k$ can be identified as a function from
$S^1$ to $S^1$. By item (3) in Theorem \ref{SobolevFacts}, we have
$\widetilde{H}^{k+1}\subseteq \widetilde{H}^k$ and
$\widetilde{\diff}(S^1)=\bigcap_k \widetilde{H}^k$.

Now we have the following embeddings:
\begin{equation}\label{Embedding}
\Diff(S^1)\subseteq \widetilde{\diff}(S^1)
\subseteq \cdots
\subseteq \widetilde{H}^3
\subseteq \widetilde{H}^2
\subseteq \widetilde{H}^1,
\end{equation}
and we can interpret Equation \eqref{eq1.6} as a sequence of
stochastic differential equations on the sequence of
affine spaces $\{\widetilde{H}^k\}_{k=1}^\infty$.

\vspace{1pt}
\subsection{The operator $\widetilde{\Phi}$ and $\Phi$}

For $\tilde{g}\in\Diff(S^1)$, let $(L_{\tilde{g}})_\ast$ be the
differential of the left translation. In accordance with the
notation used by DaPrato and Zabczyk in \cite{DaPrato}, we denote
$(L_{\tilde{g}})_\ast$ by $\widetilde{\Phi}(\tilde{g})$.

Initially,
$\widetilde{\Phi}:\Diff(S^1)\to(\diff(S^1)\to\diff(S^1))$, which
means $\widetilde{\Phi}$ takes an element $\tilde{g}\in\Diff(S^1)$
and becomes a linear transformation $\widetilde{\Phi}(\tilde{g})$
from $\diff(S^1)$ to $\diff(S^1)$ (see subsection 2.1). Because we
want to interpret Equation \eqref{eq1.6} as an SDE on
$\widetilde{H}^k$ and use DaPrato and Zabczyk's theory
\cite{DaPrato}, we need the operator $\widetilde{\Phi}$ to be
extended as $\widetilde{\Phi}:\widetilde{H}^k\to(H_\lambda\to H^k)$,
which means $\widetilde{\Phi}$ takes an element
$\tilde{g}\in\widetilde{H}^k$ and becomes a linear transformation
$\widetilde{\Phi}(\tilde{g})$ from $H_\lambda$ to $H^k$
\cite{DaPrato}.

Let $L(H_\lambda,H^k)$ be the space of linear transformations from $H_\lambda$ to $H^k$.
Define a mapping
\begin{equation}\label{AffinePhi1}
\widetilde{\Phi}:\widetilde{C}^k\to L(H_\lambda,H^k)
\end{equation}
such that if $\tilde{f}\in \widetilde{C}^k$, $g\in H_\lambda$,
then $\widetilde{\Phi}(\tilde{f})(g)=g\circ\tilde{f}$.
The mapping $\widetilde{\Phi}$ is easily seen to be well defined.
Sometimes, it is easier to work with the vector space $C^k$.
So we similarly define a mapping
\begin{equation}\label{Phi1}
\Phi:C^k\to L(H_\lambda,H^k)
\end{equation}
such that if $f\in C^k$, $g\in H_\lambda$, then $\Phi(f)(g)=g\circ\tilde{f}$,
where $\tilde{f}=id+f$ is the image of $f$ under the isomorphism $\sim$.

Let $L^2(H_\lambda,H^k)$ denote the space of Hilbert-Schmidt operators from $H_\lambda$ to $H^k$.
The space $L^2(H_\lambda,H^k)$ is a separable Hilbert space.
For $T\in L^2(H_\lambda,H^k)$, the norm of $T$ is given by
\[
\|T\|_{L^2(H_\lambda,H^k)}^2=\sum_{n\in\mathbb{Z}} \|T\hat{e}_n^{(\lambda)}\|_{H^k}^2
\]
where $\hat{e}_n^{(\lambda)}$ is defined in Definition (\ref{scaling}).

To use DaPrato and Zabczyk's theory \cite{DaPrato}, we need $\widetilde{\Phi}$ to be
$\widetilde{\Phi}:\widetilde{H}^k\to L^2(H_\lambda,H^k)$ or equivalently,
we need $\Phi$ to be $\Phi:H^k\to L^2(H_\lambda,H^k)$.
We will also need some Lipschitz condition of $\widetilde{\Phi}$ and $\Phi$.
These are proved in proposition (\ref{HS}) and (\ref{LocalLip}).
Both propositions need the Fa\`{a} di Bruno's formula for higher derivatives of a
composition function.

\begin{theorem}[Fa\`{a} di Bruno's formula \cite{Bruno}]\label{Bruno}

\begin{equation}\label{BrunoFormula}
f(g(x))^{(n)}= \sum_{k=0}^n f^{(k)}(g(x))
B_{n,k}(g'(x),g''(x),\cdots,g^{(n-k+1)}(x)),
\end{equation}
where $B_{n,k}$ is the Bell polynomial
\[
B_{n,k}(x_1,\cdots,x_{n-k+1})= \sum \frac{n!}{j_1!\cdots j_{n-k+1}!}
\Big(\frac{x_1}{1!}\Big)^{j_1}\cdots
\Big(\frac{x_{n-k+1}}{(n-k+1)!}\Big)^{j_{n-k+1}},
\]
and the summation is taken over all sequences of
$\{j_1,\cdots,j_{n-k+1}\}$ of nonnegative integers such that
$j_1+\cdots+j_{n-k+1}=k$ and $j_1+2j_2+\cdots+(n-k+1)j_{n-k+1}=n$.
\end{theorem}

We remark that after expanding expression \eqref{BrunoFormula},
$f(g(x))^{(n)}$ can be viewed as a summation of several terms,
each of which has the form
\[
f^{(j)}(g(x))m(g',g'',\cdots,g^{(n)})
\]
where $j\le n$ and $m(g',g'',\cdots,g^{(n)})$ is a \emph{monomial}
in $g',g'',\cdots,g^{(n)}$.
Also observe that, the only term that involves the highest derivative
of $g$ is $f'(g(x))g^{(n)}(x)$.

\begin{proposition}\label{HS}
For any $f\in C^k$, $k=0,1,2,\cdots$, $\Phi(f)\in L^2(H_\lambda,H^k)$.
\end{proposition}

\begin{proof}
\begin{align*}
\|\Phi(f)\|_{L^2(H_\lambda,H^k)}^2 & =\sum_{n\in\mathbb{Z}} \|\Phi(f)(\hat{e}_n)\|_{H^k}^2\\
&=\sum_{n\in\mathbb{Z}} \|\hat{e}_n(id+f)\|_{L^2}^2 + \|\hat{e}_n(id+f)^{(k)}\|_{L^2}^2,
\end{align*}
where $\hat{e}_n$ is defined in Definition (\ref{scaling}) and we have suppressed
the index $\lambda$ here.
$\hat{e}_n(id+f)$ denotes the function $\hat{e}_n$ composed with
$id+f$, and $\hat{e}_n(id+f)^{(k)}$ is the $k$th derivative of
$\hat{e}_n(id+f)$.

First, we have
\[
\|\hat{e}_n(id+f)\|_{L^2}^2\le \lambda(n)^2.
\]
We apply Fa\`{a} di Bruno's formula \eqref{BrunoFormula} to
$\hat{e}_n(id+f)^{(k)}$, and then expand it to a summation of
several terms. We are going to deal with the terms with and without
$f^{(k)}$, the highest derivative of $f$, separately.
So we write the summaion as
\begin{equation}\label{eq2.11}
\hat{e}_n(id+f)^{(k)}=...\textrm{ terms without } f^{(k)}... +
\hat{e}_n'(id+f)f^{(k)},
\end{equation}
where each term without $f^{(k)}$ has the form
\[
\hat{e}_n^{(j)}(id+f)m(f',f'',\cdots,f^{(k-1)})
\]
with $j\le k$ and $m(f',f'',\cdots,f^{(k-1)})$ a \emph{monomial} in
$f',f'',\cdots,f^{(k-1)}$. Let $d$ be the degree of the
monomial $m(f',f'',\cdots,f^{(k-1)})$. Then from Fa\`{a} di Bruno's formula
we see that $d\le k$ for all monomials.

By Definition \ref{scaling} of $\hat{e}_n$ and using item (\ref{Sobolev2}) in
Theorem \ref{SobolevFacts}, we have
\begin{align}
&\|\hat{e}_n^{(j)}(id+f)m(f',f'', \cdots, f^{(k-1)})\|_{L^2} \notag \\
&\le \|\hat{e}_n^{(j)}(id+f)\|_{L^\infty}
\|m(f',f'',\cdots,f^{(k-1)})\|_{L^\infty} \label{eq2.12}
\\
& \le \lambda(n)|n|^k c_k^k\|f\|_{H^k}^k. \notag
\end{align}

For the last term in expression \eqref{eq2.11}, we have
\begin{align}
\|\hat{e}_n'(id+f)f^{(k)}\|_{L^2}
&\le \|\hat{e}_n'(id+f)\|_{L^\infty}\|f^{(k)}\|_{L^2}\label{eq2.13} \\
&\le \lambda(n)|n| \|f\|_{H^k} \le \lambda(n)|n|^k
c_k^k\|f\|_{H^k}^k. \notag
\end{align}

By \eqref{eq2.12} and \eqref{eq2.13}, we have

\[
\|\hat{e}_n(id+f)^{(k)}\|_{L^2}^2 \le K \lambda(n)^2 |n|^{2k}
c_k^{2k}\|f\|_{H^k}^{2k},
\]
where $K$ is the number of terms in expression \eqref{eq2.11}, which
depends on $k$ but does not depend on $n$.
Therefore,
\[
\|\Phi(f)\|_{L^2(H_\lambda,H^k)}^2 \le \sum_{n\in\mathbb{Z}} \left(\lambda(n)^2 + K
\lambda(n)^2 |n|^{2k} c_k^{2k}\|f\|_{H^k}^{2k}\right)
\]
Because $\lambda(n)$ is rapidly decreasing (Definition \ref{scaling}),
$\sum_{n\in\mathbb{Z}}\lambda(n)^2 |n|^{2k}<\infty$.
Therefore, we have
\[
\|\Phi(f)\|_{L^2(H_\lambda,H^k)}^2 < \infty
\]

\end{proof}

Now $\Phi$ can be viewed as a mapping $\Phi:C^k\to
L^2(H_\lambda,H^k)$. Similarly, $\widetilde{\Phi}$ can be viewed as
a mapping $\widetilde{\Phi}:\widetilde{C}^k\to L^2(H_\lambda,H^k)$.
To use DaPrato and Zabczyk's theory \cite{DaPrato}, we will need
the Lipschitz condition of $\Phi$ and $\widetilde{\Phi}$.
It turns out that they are \emph{locally} Lipschitz.
Let us recall the concept of local Lipschitzness: Let $A$ and $B$ be
two normed linear spaces with norm $\|\cdot\|_A$ and $\|\cdot\|_B$
respectively. A mapping $f:A\to B$ is said to be \emph{locally
Lipschitz} if for $R>0$, and $x,y\in A$ such that $\|x\|,\|y\|\le
R$, we have
\[
\|f(x)-f(y)\|_B \le C_{R} \|x-y\|_A,
\]
where $C_N$ is a constant which in general depends on $N$.

\begin{proposition}\label{LocalLip}
For any $k=0,1,2,\cdots$, $\Phi:C^k\to L^2(H_\lambda,H^k)$ is locally Lipschitz.
\end{proposition}

\begin{proof}
Let $R>0$, and $f,g\in C^k$ be such that $\|f\|_{H^k},\|g\|_{H^k}\le
R$. We have
\begin{align*}
&\|\Phi(f)-\Phi(g)\|_{L^2(H_\lambda,H^k)}^2 \\
&=\sum_{n\in\mathbb{Z}} \|[\Phi(f)-\Phi(g)]\hat{e}_n\|_{H^k}^2
=\sum_{n\in\mathbb{Z}} \|\hat{e}_n(id+f)-\hat{e}_n(id+g)\|_{H^k}^2\\
&=\sum_{n\in\mathbb{Z}} \|\hat{e}_n(id+f)-\hat{e}_n(id+g)\|_{L^2}^2 +
\|\hat{e}_n(id+f)^{(k)}-\hat{e}_n(id+g)^{(k)}\|_{L^2}^2,
\end{align*}
where $\hat{e}_n$ is defined in Definition (\ref{scaling}) and we have suppressed
the index $\lambda$ here.
$\hat{e}_n(id+f)$ and $\hat{e}_n(id+g)$ denote the function
$\hat{e}_n$ composed with $id+f$ and $id+g$ respectively.
$\hat{e}_n(id+f)^{(k)}$ and $\hat{e}_n(id+g)^{(k)}$ are the $k$th
derivatives of $\hat{e}_n(id+f)$ and $\hat{e}_n(id+g)$ respectively.

First, by the mean value theorem we have
\begin{align}
&\|\hat{e}_n(id+f)-\hat{e}_n(id+g)\|_{L^2}
=\|\hat{e}_n'(id+\xi)(f-g)\|_{L^2}\\
&\le \|\hat{e}_n'(id+\xi)\|_{L^\infty}\|f-g\|_{L^2}
\le \lambda(n) |n| \|f-g\|_{H^k}
\end{align}
We apply Fa\`{a} di Bruno's formula \eqref{BrunoFormula} to
$\hat{e}_n(id+f)^{(k)}$, and then expand it to a summation of
several terms. We are going to deal with the terms with and without
$f^{(k)}$, the highest derivative of $f$, separately.
So we write the summaion as
\begin{equation}\label{eq2.15}
\hat{e}_n(id+f)^{(k)}=...\textrm{ terms without } f^{(k)}... +
\hat{e}_n'(id+f)f^{(k)},
\end{equation}
where each term without $f^{(k)}$ has the form
\[
\hat{e}_n^{(j)}(id+f)m(f',f'',\cdots,f^{(k-1)})
\]
with $j\le k$ and $m(f',f'',\cdots,f^{(k-1)})$ a \emph{monomial} in
$f',f'',\cdots,f^{(k-1)}$.
Let $d$ be the degree of the
monomial $m(f',f'',\cdots,f^{(k-1)})$. Then from Fa\`{a} di Bruno's formula
we see that $d\le k$ for all monomials.
By replacing $f$ with
$g$ in \eqref{eq2.15}, we obtain
\begin{equation}
\hat{e}_n(id+g)^{(k)}=...\textrm{ terms without } g^{(k)}...
+ \hat{e}_n'(id+g)g^{(k)}
\end{equation}

Next, we need a simple observation: suppose $A_1A_2A_3...$ and
$B_1B_2B_3...$ are two monomials with the same number of factors. By
telescoping, we can put $A_1A_2A_3...-B_1B_2B_3...$ into the form
\[
(A_1-B_1)A_2A_3...+B_1(A_2-B_2)A_3...+B_1B_2(A_3-B_3)...+\cdots
\]
Using this observation, we can put $\hat{e}_n(id+f)^{(k)}-\hat{e}_n(id+g)^{(k)}$ into the form
\begin{align}
\hat{e}_n(id+f)^{(k)}&-\hat{e}_n(id+g)^{(k)}=
...\textrm{terms without } f^{(k)} \textrm{ and } g^{(k)}... \label{eq2.17} \\
&+\left(\hat{e}_n'(id+f)-\hat{e}_n'(id+g)\right)f^{(k)}+\hat{e}_n'(id+g)\left(f^{(k)}-g^{(k)}\right)
\notag
\end{align}
In expression \eqref{eq2.17},
there are two types of terms without $f^{(k)}$ and $g^{(k)}$.
One type has the form
\begin{equation}
\left(\hat{e}_n^{(j)}(id+f)-\hat{e}_n^{(j)}(id+g)\right)m_A(f',\cdots,f^{(k-1)},g',\cdots,g^{(k-1)}),
\end{equation}
where $j\le k$ and $m_A$ is a monomial in
$f',\cdots,f^{(k-1)},g',\cdots,g^{(k-1)}$. We denote such a term by
$A$. Another type has the form
\begin{equation}
\hat{e}_n^{(i)}(id+g)\left(f^{(j)}-g^{(j)}\right)m_B(f',\cdots,f^{(k-1)},g',\cdots,g^{(k-1)})
\end{equation}
where $i,j\le k$ and $m_B$ is a monomial in
$f',\cdots,f^{(k-1)},g',\cdots,g^{(k-1)}$. We denote such a term by
$B$.

Now we want to find an $L^2$ bound of each term in \eqref{eq2.17}.
For the term $A$, by the mean value theorem we have
\[
[\hat{e}_n^{(j)}(id+f)-\hat{e}_n^{(j)}(id+g)]
=\hat{e}_n^{(j+1)}(id+\xi)(f-g).
\]
By Definition \ref{scaling} of $\hat{e}_n$,
and using Item (\ref{Sobolev1}) and (\ref{Sobolev2})
in Theorem \ref{SobolevFacts}, we have
\begin{align}
\|A\|_{L^2}
&\le \|\hat{e}_n^{(j+1)}(id+\xi)\|_{L^\infty}
\|m_A\|_{L^\infty}
\|f-g\|_{L^2}\label{eq2.20} \\
&\le \lambda(n)|n|^{k+1}c_k^kN^k \|f-g\|_{H^k}. \notag
\end{align}
For the term $B$, we have
\begin{align}
\|B\|_{L^2}
&\le \|\hat{e}_n^{(i)}(id+g)\|_{L^\infty}
\|m_B\|_{L^\infty}
\|f^{(j)}-g^{(j)}\|_{L^2}\label{eq2.21} \\
&\le \lambda(n)|n|^{k}c_k^kN^k \|f-g\|_{H^k}. \notag
\end{align}

For the last two terms in expression \eqref{eq2.17},
using Item (\ref{Sobolev1}) and (\ref{Sobolev2})
in Theorem \ref{SobolevFacts} again, we have
\begin{align}
&\|[\hat{e}_n'(id+f)-\hat{e}_n'(id+g)]f^{(k)}\|_{L^2} \notag\\
&=\|\hat{e}_n''(id+\xi)(f-g)f^{(k)}\|_{L^2}
\le\|\hat{e}_n''(id+\xi)\|_{L^\infty}\|f-g\|_{L^\infty}\|f^{(k)}\|_{L^2} \label{eq2.22} \\
&\le\|\hat{e}_n''(id+\xi)\|_{L^\infty}c_k\|f-g\|_{H^k}\|f\|_{H^k}
\le \lambda(n) |n|^2 c_kN \|f-g\|_{H^k} \notag
\end{align}
and
\begin{equation}\label{eq2.23}
\|\hat{e}_n'(id+g)[f^{(k)}-g^{(k)}]\|_{L^2} \le \lambda(n) |n|
\|f-g\|_{H^k}.
\end{equation}

By \eqref{eq2.20}--\eqref{eq2.23}, we see that
$\lambda(n)|n|^{k+1}c_k^kN^k\|f-g\|_{H^k}$ is a common $L^2$ bound
for all terms in \eqref{eq2.17}. So,
\begin{equation}
\|\hat{e}_n(id+f)^{(k)}-\hat{e}_n(id+g)^{(k)}\|_{L^2}
\le K\lambda(n)|n|^{k+1}c_k^kN^k\|f-g\|_{H^k}
\end{equation}
where $K$ is the number of terms in expression \eqref{eq2.17},
which depends on $k$ but does not depend on $n$.

Finally,
\begin{align*}
&\|\Phi(f)-\Phi(g)\|_{L^2(H_\lambda,H^k)}^2\\
&\le \sum_{n\in\mathbb{Z}} \lambda(n)^2 |n|^2 \|f-g\|_{H^k}^2
+ K^2 \lambda(n)^2 |n|^{2k+2} c_k^{2k} R^{2k} \|f-g\|_{H^k}^2\\
&\le
Kc_k^kR^k\|f-g\|_{H^k}\left(\sum_{n\in\mathbb{Z}}\lambda(n)^2|n|^{2k+2}\right)^{1/2}
\end{align*}
Let
\[
C_R=\left(\sum_{n\in\mathbb{Z}} \lambda(n)^2 |n|^2+K^2 \lambda(n)^2 |n|^{2k+2}
c_k^{2k} R^{2k}\right)^{1/2},
\]
Because $\lambda(n)$ is rapidly decreasing (Definition \ref{scaling}),
$\sum_{n\in\mathbb{Z}}\lambda(n)^2 |n|^{2k}<\infty$.
So $C_R$ is a finite number that depends on $R$ and $k$.
Therefore,
\begin{equation}
\|\Phi(f)-\Phi(g)\|_{L^2(H_\lambda,H^k)}\le C_{R} \|f-g\|_{H^k}.
\end{equation}

\end{proof}

By the above proposition, $\Phi:C^k\to L^2(H_\lambda,H^k)$ is locally Lipschitz.
So $\Phi$ is uniformly continuous on $C^k$.
But $C^k$ is a dense subspace of $H^k$ (see subsection 2.3).
Therefore, we can extend the domain of $\Phi$ from $C^k$ to $H^k$, and
obtain a mapping $\Phi:H^k\to L^2(H_\lambda,H^k)$.
Similarly, we can also extend the domain of $\widetilde{\Phi}$ from
$\widetilde{C}^k$ to $\widetilde{H}^k$, and obtain a mapping
$\widetilde{\Phi}:\widetilde{H}^k\to L^2(H_\lambda,H^k)$.
After extension, $\Phi$ and $\widetilde{\Phi}$ are still locally Lipschitz.

\begin{definition}\label{defPhi}
Define
$\widetilde{\Phi}:\widetilde{H}^k\to L^2(H_\lambda,H^k)$
to be the extension of
$\widetilde{\Phi}:\widetilde{C}^k\to L^2(H_\lambda,H^k)$
from $\widetilde{C}^k$ to $\widetilde{H}^k$,
and
$\Phi:H^k\to L^2(H_\lambda,H^k)$
to be the extension of
$\Phi:C^k\to L^2(H_\lambda,H^k)$
from $C^k$ to $H^k$.
By the remark in the previous paragraph, $\Phi$ and $\widetilde{\Phi}$
are still locally Lipschitz.
\end{definition}

\section{The main result}\label{section3}

In this section, we fix a probability space $(\Omega, \mathcal{F}, \mathbb{P})$
equipped with a filtration $\mathcal{F}_\ast=\{\mathcal{F}_t, t\ge0\}$
that is right continuous and
such that each $\mathcal{F}_t$ is complete with respect to $\mathbb{P}$.

Equation \eqref{eq1.6} is now interpreted as a Stratonovich
stochastic differential equation on $\widetilde{H}^k$ for each $k=0,1,2,\cdots$.
Let us fix such a $k$.

\vspace{1pt}
\subsection{Changing Equation \eqref{eq1.6} into the It\^o form}

To solve Equation \eqref{eq1.6}, we first need to change it into the It\^o form.
Here we follow the treatment of S. Fang in \cite{Fang2002}.
In Definition \ref{Wt}, $W_t=\sum_{n\in\mathbb{Z}}
B_t^{(n)}\hat{e}_n^{(\alpha)}$, where $\alpha$ is a rapidly
decreasing function as described in Definition \ref{scaling}. Using
the definition of $\widetilde{\Phi}$, $W_t$, and $\hat{e}_n^{(\alpha)}$, we can
write Equation \eqref{eq1.6} as
\begin{equation}\label{eq3.1}
\delta\widetilde{X}_t=\alpha(0) +\sum_{n=1}^\infty
\alpha(n)\cos(n\widetilde{X}_t)\delta B_t^{(n)} +\sum_{m=1}^\infty
\alpha(m)\sin(m\widetilde{X}_t)\delta B_t^{(m)}.
\end{equation}
Using the stochastic contraction of $dB_t^{(n)}\cdot dB_t^{(m)}=\delta_{mn}dt$, we have
\begin{align*}
\alpha(n)d\cos(n\widetilde{X}_t)\cdot dB_t^{(n)}
&= -\alpha(n)^2\sin(n\widetilde{X}_t)\cos(n\widetilde{X}_t)dt\\
\alpha(n)d\sin(m\widetilde{X}_t)\cdot dB_t^{(m)}
&= \alpha(m)^2\sin(m\widetilde{X}_t)\cos(m\widetilde{X}_t)dt
\end{align*}
So the stochastic contraction of the right hand side of
\eqref{eq3.1} is zero. Therefore Equation \eqref{eq3.1} can be
written in the following It\^o form:
\begin{equation}\label{eq3.2}
d\widetilde{X}_t=\alpha(0)
+\sum_{n=1}^\infty \alpha(n)\cos(n\widetilde{X}_t)d B_t^{(n)}
+\sum_{m=1}^\infty \alpha(m)\sin(m\widetilde{X}_t)d B_t^{(m)}
\end{equation}
Using the definition of $W_t$ and $\widetilde{\Phi}$ again,
Equation \eqref{eq3.2} becomes
\begin{equation}
d\widetilde{X}_t=\widetilde{\Phi}(\widetilde{X}_t)dW_t
\end{equation}
Therefore, Equation \eqref{eq1.6} is equivalent to the following It\^o stochastic
differential equation
\begin{equation}\label{affineSDE}
d\widetilde{X}_t=\widetilde{\Phi}(\widetilde{X}_t)dW_t,
\hspace{.2in}
\widetilde{X}_0=id
\end{equation}
This equation is considered in the affine space $\widetilde{H}^k$.

If we write $\widetilde{X}_t=id+X_t$ with $X_t$
a process with values in the Sobolev space $H^k$
and use the definition of $\Phi$ (see subsection 2.4),
Equation \eqref{affineSDE} is equivalent to the following equation
\begin{equation}\label{fullSDE}
dX_t=\Phi(X_t)dW_t,
\hspace{.2in}
X_0=0
\end{equation}
This equation is considered in the Sobolev space $H^k$.

\vspace{1pt}
\subsection{Truncated stochastic differential equation}

By Proposition (\ref{LocalLip}) the operator $\Phi$ is locally Lipschitz.
To use G. DaPrato and J. Zabczyk's theory \cite{DaPrato}, we need to ``truncate''
the operator $\Phi$:
Let $R>0$.
Let $\Phi_R:H^k\to L^2(H_\alpha,H^k)$ be defined by
\begin{equation}
\Phi_R(x)=\left\{
\begin{array}{ll}
\Phi(x), & \|x\|_{H^k}\le R\\
\Phi(Rx/\|x\|_{H^k}), & \|x\|_{H^k}>R
\end{array}
\right.
\end{equation}
Then $\Phi_R$ is globally Lipschitz.
Let us consider the following ``truncated'' stochastic differential equation
\begin{equation}\label{truncatedSDE}
dX_t=\Phi_R(X_t)dW_t,
\hspace{.2in}
X_0=0
\end{equation}
in the Sobolev space $H^k$.
The following defintion is in accordance with G. DaPrato and J. Zabczyk's
treatments (p.182 in \cite{DaPrato}).

\begin{definition}
Let $T>0$.
An $\mathcal{F}_\ast$-adapted $H^k$-valued process $X_t$ with continuous sample paths
is said to be a mild solution to Equation \eqref{truncatedSDE} up to time $T$ if
\[
\int_0^T\|X_s\|_{H^k}^2 ds < \infty,
\hspace{.2in}
\mathbb{P}\textrm{-a.s.}
\]
and for all $t\in[0,T]$, we have
\[
X_t=X_0+\int_0^t \Phi_R(X_s)dW_s,
\hspace{.2in}
\mathbb{P}\textrm{-a.s.}
\]
For Equation \eqref{truncatedSDE}, a strong solution is the same as a mild solution.
The solution $X_t$ is said to be unique up to time $T$ if for any other
solution $Y_t$, the two processes $X_t$ and $Y_t$ are equivalent up to time $T$,
that is, the stopped processes $X_{t\wedge T}$ and $Y_{t\wedge T}$ are equivalent.
\end{definition}

\begin{remark}
In the above definition, we require a solution to have continuous sample paths.
\end{remark}

\begin{proposition}
For each $T>0$, there is a unique solution $X^{(T)}$ to Equation \eqref{truncatedSDE}
up to time $T$.
\end{proposition}

\begin{proof}
The proof is a simple application of Theorem 7.4, p.186 from \cite{DaPrato}.
We need to check the conditions to use Theorem 7.4 from \cite{DaPrato}.
By definition of $\Phi_R$, we see that $\Phi_R$ satisfies the following growth condition:
\[
\|\Phi_R(x)\|_{L^2(H_\alpha,H^k)}^2 \le C (1+\|x\|_{H^k}^2),
\hspace{.2in}x\in H^k
\]
for some constant $C$.
All other conditions to use Theorem 7.4 from \cite{DaPrato} are easily verified.
Therefore, we have the conclusion.

\end{proof}

Let us choose a sequence $\{T_n\}_{n=1}^\infty$ such that $T_n \uparrow \infty$,
and let each $X^{(T_n)}$ be the unique solution to Equation \eqref{truncatedSDE}
up to time $T_n$.
By the uniqueness of the solution, and by the continuity of sample paths,
for $1\le i<j$, the sample paths of $X^{(T_j)}$ coincide with the sample paths
of $X^{(T_i)}$ up to time $T_i$ almost surely.
To be precise, we have, for almost all $\omega\in\Omega$,
\[
X^{(T_j)}(t,\omega)=X^{(T_i)}(t,\omega),
\hspace{.2in}
\textrm{for all } t\in[0,T_i]
\]
Therefore, we can extend the sample paths to obtain a process $X^R$:
For almost all $\omega\in\Omega$, let
\[
X^R(t,\omega)=\lim_{n\to\infty} X^{(T_n)}(t,\omega)
\hspace{.2in}
\textrm{for all } t\in[0,\infty)
\]
Then the process $X^R$ is a unique solution with continuous sample paths
to Equation \eqref{truncatedSDE} up to time $T$ for all $T>0$.

\begin{remark}
The above construction of the process $X^R$ is independent of the choice of the sequence
$\{T_n\}_{n=1}^\infty$:
Let $\{S_n\}_{n=1}^\infty$ be another sequence such that $S_n \uparrow \infty$.
Let $Y^R$ be the process contructed as above but using the sequence
$\{S_n\}_{n=1}^\infty$.
Then $X^R$ and $Y^R$ are equivalent up to $T$ for all $T>0$.
Therefore, they are equivalent.
\end{remark}

\begin{definition}\label{XR}
For every $R>0$, we define $X^R$ to be the $H^k$-valued process
with continuous sample paths as constructed above.
Define
\begin{equation}
\tau_R=\inf\{t: \|X^R(t)\|_{H^k}\ge R\}
\end{equation}
\end{definition}

\vspace{1pt}
\subsection{Solutions up to stopping times}

Let us consider Equation \eqref{fullSDE} in the Sobolev space $H^k$.
The following definition is in accordance with E. Hsu's treatments in \cite{Hsu}.

\begin{definition}\label{HsuDef1}
Let $\tau$ be an $\mathcal{F}_\ast$-stopping time.
An $\mathcal{F}_\ast$-adapted process $X_t$ with continuous sample paths is said to
be a solution to Equation \eqref{fullSDE} up to time $\tau$ if for all $t\ge0$
\[
X_{t\wedge\tau}=X_0+\int_0^{t\wedge\tau}\Phi(X_s)dW_s
\]
The solution $X_t$ is said to be unique up to $\tau$ if
for any other solution $Y_t$, the two processes $X_t$ and $Y_t$ are
equivalent up to $\tau$, that is, the stopped processes $X_{t\wedge\tau}$
and $Y_{t\wedge\tau}$ are equivalent.
\end{definition}

\begin{remark}
We can similarly define an $\widetilde{H}^k$-valued process being the unique solution
to Equation \eqref{affineSDE} up to a stopping time $\tau$.
Clearly, we have the following:
If $X_t$ is the solution to Equation \eqref{fullSDE} up to a stopping time $\tau$,
then the $\widetilde{H}^k$-valued process $\widetilde{X}_t=id+X_t$ is the solution
to Equation \eqref{affineSDE} up to time $\tau$ and vice versa.
\end{remark}

\begin{remark}
If $X_t$ is a solution to Equation \eqref{fullSDE} up to $\tau$,
then it is also a solution up to $\sigma$ for any $\mathcal{F}_\ast$-stopping time
$\sigma$ such that $\sigma\le\tau$ a.s.
\end{remark}

\begin{proposition}
Let $R>0$. Let $X^R$ and $\tau_R$ be defined as in Definition (\ref{XR}).
Then $X^R$ is the unique solution to Equation \eqref{fullSDE} up to $\tau_R$.
\end{proposition}

\begin{proof}
Because $X^R$ is the unique solution to Equation \eqref{truncatedSDE} up to $T$
for all $T>0$, we have
\[
X^R_t=\int_0^t \Phi_R(X^R_s)dW_s
\]
for all $t\ge0$.
By the definition of $\Phi_R$, we have $\Phi_R(X^R_s)=\Phi(X^R_s)$ for $s\le\tau_R$.
So,
\[
X^R_{t\wedge\tau_R}=\int_0^{t\wedge\tau_R}\Phi_R(X^R_s)dW_s
=\int_0^{t\wedge\tau_R}\Phi(X^R_s)dW_s
\]
Therefore, $X^R$ is a solution to Equation \eqref{fullSDE} up to $\tau_R$.

Suppose $Y_t$ is another solution to Equation \eqref{fullSDE} up to $\tau_R$.
Then $Y_t$ is also a solution to Equation \eqref{truncatedSDE} up to $\tau_R$.
But $X^R_t$ is the unique solution to Equation \eqref{truncatedSDE} up to $T$ for all $T>0$.
Therefore, $Y_t$ and $X^R_t$ are equivalent up to $\tau_R$.

\end{proof}

Let us choose a sequence $\{R_n\}_{n=1}^\infty$ such that $R_n \uparrow \infty$,
and let $X^{R_n}$ and $\tau_{R_n}$ be defined as in Definition (\ref{XR}).
For $1\le i<j$, we have $\Phi_{R_i}(x)=\Phi_{R_j}(x)$ for $\|x\|_{H^k}\le R_i$.
Thus, $X^{R_j}$ is also a solution to Equation \eqref{fullSDE} up to $\tau_{R_i}$.
Therefore, by the uniqueness of solution and by the continuity of sample paths of solution,
the sample paths of $X^{R_j}$ coincide with the sample paths of $X^{R_i}$ almost surely.
To be precise, we have, for almost all $\omega\in\Omega$,
\[
X^{R_j}(t,\omega)=X^{R_i}(t,\omega),
\hspace{.2in}
\textrm{for all } t\in[0,\tau_{R_i}(\omega)]
\]
Consequently, $\{\tau_{R_n}\}_{n=1}^\infty$ is an increasing sequence of stopping times.
Let
\begin{equation}
\tau_\infty=\lim_{n\to\infty} \tau_{R_n}
\end{equation}
Now we can extend the sample paths of $X^{R_n}$ to obtain a process $X^\infty$:
For almost all $\omega\in\Omega$, let
\[
X^\infty(t,\omega)=\lim_{n\to\infty} X^{R_n}(t,\omega)
\hspace{.2in}
\textrm{for all } 0\le t<\tau_\infty(\omega)
\]
Then the process $X^\infty$ is a unique solution with continuous sample paths
to Equation \eqref{fullSDE} up to time $\tau_R$ for all $R>0$.
Also, the stopping time $\tau_R$ defined in Definition (\ref{XR}) is realized by
the process $X^\infty$:
\[
\tau_R=\inf\{t: \|X^\infty(t)\|_{H^k}\ge R\}
\]

\begin{remark}
The above constructions of
the process $X^\infty$ 
and
the stopping time $\tau_\infty$
are independent of the choice of the sequence $\{R_n\}_{n=1}^\infty$:
Let $\{S_n\}_{n=1}^\infty$ be another sequence such that $S_n \uparrow \infty$.
Let $\sigma_\infty$ be the stopping time and $Y^\infty$ be the process
contructed as above but using the sequence $\{S_n\}_{n=1}^\infty$.
First, we can combine the two sequences $\{R_n\}_{n=1}^\infty$ and $\{S_n\}_{n=1}^\infty$
to form a new sequence $\{K_n\}_{n=1}^\infty$ such that $K_n \uparrow \infty$.
Let $\gamma_\infty$ be the stopping time constructed as above but using the sequence
$\{K_n\}_{n=1}^\infty$.
Then $\tau_\infty=\sigma_\infty=\gamma_\infty$.
Also, $X^\infty$ and $Y^\infty$ are equivalent up to $\tau_{R_n}$ and $\tau_{S_n}$
for all $n=1,2,\cdots$. Therefore, they are equivalent up to $\tau_\infty$.
\end{remark}

\begin{definition}\label{XInf}
We define $X^\infty$ to be the $H^k$-valued process and $\tau_\infty$ to be the stopping time
as constructed above. We call $\tau_\infty$ the explosion time of the process $X^\infty$.
We also define the $\widetilde{H}^k$-valued process $\widetilde{X}^\infty$ to be
$\widetilde{X}^\infty=id+X^\infty$.
\end{definition}

We can slightly extend Definition (\ref{HsuDef1}) and make the following definition:

\begin{definition}\label{HsuDef2}
Let $\tau$ be an $\mathcal{F}_\ast$-stopping time.
An $\mathcal{F}_\ast$-adapted process $X_t$ with continuous sample paths is said to
be a solution to Equation \eqref{fullSDE} up to time $\tau$ if
there is an increasing sequence of $\mathcal{F}_\ast$-stopping time
$\{\tau_n\}_{n=1}^\infty$ such that $\tau_n \uparrow \tau$
and $X_t$ is a solution to Equation \eqref{fullSDE} up to time $\tau_n$ in the sense
of Definition (\ref{HsuDef1}) for all $n=1,2,\cdots$.
The solution $X_t$ is said to be unique up to $\tau$ if
it is unique up to $\tau_n$ for all $n=1,2,\cdots$.
\end{definition}

We have proved the following proposition:

\begin{proposition}\label{part1}
Let $k$ be a non-negative integer.
The process $X^\infty$ as defined in Definition (\ref{XInf}) is the unique solution with continuous
sample paths to Equation \eqref{fullSDE} up to the explosion time $\tau_\infty$.
\end{proposition}

\vspace{1pt}
\subsection{The main result}

In this subsection, we will prove that the explosion time $\tau_\infty$ defined in
Definition (\ref{XInf}) is infinity almost surely. We will also prove that
the process $\widetilde{X}^\infty$ defined in Definition (\ref{XInf})
lives in the group $\Diff(S^1)$.
The key idea to both proofs is the following proposition:

\begin{proposition}\label{Sasha}
Let $\widetilde{X}_t$ be an $\mathcal{F}_\ast$-adapted $\widetilde{H}^k$-valued process
with continuous sample paths and $\tau$ an $\mathcal{F}_\ast$-stopping time.
If $\widetilde{X}_t$ is a solution to
\[
d\widetilde{X}_t=\widetilde{\Phi}(\widetilde{X}_t)dW_t,
\hspace{.2in}
\widetilde{X}_0=id
\]
up to $\tau$, then $\widetilde{X}_t\circ\tilde{\xi}$ is a solution to
\[
d\widetilde{X}_t=\widetilde{\Phi}(\widetilde{X}_t)dW_t,
\hspace{.2in}
\widetilde{X}_0=\tilde{\xi}
\]
up to $\tau$, where $\tilde{\xi}$ is a bounded $\widetilde{H}^k$-valued random variable
and ``$\circ$'' is the composition of two functions.
\end{proposition}

\begin{proof}
By assumption
\[
\widetilde{X}_{t\wedge\tau}=id+\int_0^{t\wedge\tau}\widetilde{\Phi}(\widetilde{X}_s)dW_s
\]
By definition of the operator $\widetilde{\Phi}$ (see subsection 2.4), this can be written as
\[
\widetilde{X}_{t\wedge\tau}=id+\int_0^{t\wedge\tau}dW_s\circ\widetilde{X}_s
\]
So
\[
\widetilde{X}_{t\wedge\tau}\circ\tilde{\xi}=\tilde{\xi}+\int_0^{t\wedge\tau}dW_s\circ\widetilde{X}_s\circ\tilde{\xi}
\]
that is
\[
\widetilde{X}_{t\wedge\tau}\circ\tilde{\xi}=\tilde{\xi}+\int_0^{t\wedge\tau}\widetilde{\Phi}(\widetilde{X}_s\circ\tilde{\xi})dW_s
\]
Therefore,
$\widetilde{X}_t\circ\tilde{\xi}$ is a solution to
\[
d\widetilde{X}_t=\widetilde{\Phi}(\widetilde{X}_t)dW_t,
\hspace{.2in}
\widetilde{X}_0=\tilde{\xi}
\]
up to $\tau$.

\end{proof}

\begin{remark}\label{concatenate}
(\textbf{Concatenating procedure}.)
Let $R>0$.
Let $\tilde{\xi}=\widetilde{X}^\infty(\tau_R)$.
Then $\tilde{\xi}$ is an $\widetilde{H}^k$-valued bounded random variable.
Let $W'_t=W_{t+\tau_R}-W_{\tau_R}$.
By the strong Markov property of the Brownian motion $W_t$,
we have $W'_t=W_t$ in distribution for all $t\ge0$.
Therefore, similar to the construction of $X^\infty$ and $\widetilde{X}^\infty$,
we can construct $Y^\infty$ and $\widetilde{Y}^\infty$ with $\widetilde{Y}^\infty$
a solution to the following equation
\[
d\widetilde{X}_t=\widetilde{\Phi}(\widetilde{X}_t)dW'_t,
\hspace{.2in}
\widetilde{X}_0=id
\]
up to stopping time
\[
\tau'_R=\inf\{t: \|Y^\infty(t)\|_{H^k}\ge R\}
\]
Using the strong Markov property of the Brownian motion $W_t$ again,
we see that $\tau_R=\tau'_R$ in distribution, and they are independent with each other.
By Proposition (\ref{Sasha}), $\widetilde{Y}^\infty\circ\tilde{\xi}$
is the solution up to time $\tau'_R$ to the following equation
\[
d\widetilde{X}_t=\widetilde{\Phi}(\widetilde{X}_t)dW'_t,
\hspace{.2in}
\widetilde{X}_0=\tilde{\xi}
\]
Because $\tilde{\xi}=\widetilde{X}^\infty(\tau_R)$,
we can \emph{concatenate} the two processes $\widetilde{X}^\infty$
and $\widetilde{Y}^\infty$ to form a new process $\widetilde{Z}^\infty$ as follows:
\begin{equation}
\widetilde{Z}^\infty_t=\left\{
\begin{array}{ll}
\widetilde{X}_t^\infty, & \textrm{ for } t\le\tau_R\\
\widetilde{Y}_{t-\tau_R}^\infty\circ\tilde{\xi}, & \textrm{ for } t>\tau_R
\end{array}
\right.
\end{equation}
By the choice of $W'_t$, we see that the process $\widetilde{Z}^\infty$ is a solution
to Equation \eqref{affineSDE} up to time $\tau_R+\tau'_R$.
By the uniqueness of solution, $\widetilde{Z}^\infty$ is equivalent to $\widetilde{X}^\infty$
up to time $\tau_R+\tau'_R$.

We can carry out this ``concatenating'' procedure over and over again.
Thus, for any $n\in\mathbb{N}$, we can construct a process $\widetilde{Z}^\infty$
which is a solutionn to Equation \eqref{affineSDE} and is equivalent to
$\widetilde{X}^\infty$ up to time $\tau_R+\tau'_R+\cdots+\tau^{(n)}_R$
with $\tau_R, \tau'_R, \cdots$ being identical in distribution and mutually independent
with each other.
\end{remark}

\begin{proposition}\label{part2}
Let $\tau_\infty$ be the explosion time of the process $X^\infty$
defined as in Definition (\ref{XInf}).
Then $\tau_\infty=\infty$ almost surely.
\end{proposition}

\begin{proof}
We can carry out the above ``concatenating'' procedure as many times as we want.
Thus, for any $n\in\mathbb{N}$, we can construct a process $\widetilde{Z}^\infty$
which is a solutionn to Equation \eqref{affineSDE} and is equivalent to
$\widetilde{X}^\infty$ up to time $\tau_R+\tau'_R+\cdots+\tau^{(n)}_R$.

By the triangle inequality in $H^k$, we have
\[
\tau_R+\tau'_R+\cdots+\tau^{(n)}_R \le \tau_{nR} \le \tau_\infty,
\]
On the other hand, because $\tau_R, \tau'_R, \cdots$ have the same distributions and
are mutually independent with each other,
\[
\lim_{n\to\infty} \tau_R+\tau'_R+\cdots+\tau^{(n)}_R =\infty
\hspace{.1in}
\textrm{a.s.}
\]
Therefore, the explosion time $\tau_\infty=\infty$ almost surely.

\end{proof}

\begin{proposition}\label{intersection}
Let $X^\infty$ be the $H^k$-valued process defined in Defintion (\ref{XInf}).
Then $X^\infty$ actually lives in the space $\diff(S^1)$.
\end{proposition}

\begin{proof}
The construction of $X^\infty$ in subsection 3.3 is for a fixed $k$.
But the method is valid for all $k=0,1,2,\cdots$.
Let us denote by $X^{k,\infty}$ the $H^k$-valued process as constructed in subsection 3.3.
Because Equation \eqref{fullSDE} takes the same form in each space $H^k$, $k=0,1,2,\cdots$,
also, $H^{k+1}\subseteq H^k$,
we see that the $H^{k+1}$-valued process $X^{k+1,\infty}$ is also a solution
to Equation \eqref{fullSDE} in the space $H^k$.
By uniqueness of the solution, $X^{k+1,\infty}$ is equivalent to $X^{k,\infty}$.
Therefore, we can also say the solution $X^{k,\infty}$
to Equation \eqref{fullSDE} in the space $H^k$
is also the solution to Equation \eqref{fullSDE} in the space $H^{k+1}$.
By induction, the solution $X^{k,\infty}$ actually lives in $H^{k+i}$ for all $i=0,1,2,\cdots$.
Therefore it lives in $\bigcap_{i=0}^\infty H^{k+i}=\diff(S^1)$.

\end{proof}

By the above proposition, the $\widetilde{H}^k$-valued process $\widetilde{X}^\infty$
lives in the affine space $\widetilde{\diff}(S^1)$.
In the next proposition we will prove that $\widetilde{X}^\infty$ actually lives
in the group $\Diff(S^1)$.
The key to the proof is Proposition (\ref{Difference})
together with the ``concatenating'' procedure (remark 3.13).

\begin{proposition}\label{part3}
The process $\widetilde{X}^\infty$ defined in Definition (\ref{XInf})
lives in the group $\Diff(S^1)$.
\end{proposition}

\begin{proof}
Let us fix a $k\ge 2$.
Suppose $\tilde{f}\in\widetilde{H}^k$.
By item (2) in Theorem 2.5, $\|f'\|_{L^\infty}\le c_k\|f\|_{H^k}$.
Thus, by controling the $H^k$-norm of $f$ we can control the $L^\infty$-norm of $f'$.
When $\|f'\|_{L^\infty}<1$, we have $f'>-1$, or equivalently, $\tilde{f}'>0$.
If we also know that $\tilde{f}$ is $C^\infty$, then by Proposition (\ref{Difference}),
we can conclude that $\tilde{f}$ is actually a diffeomorphism of $S^1$.

The process $X^\infty$ has values in the $R$-ball
\[
B(0,R)=\{x\in H^k : \|x\|_{H^k}\le R\}
\]
up to time $\tau_R$.
Let us choose $R$ so that $f\in B(0,R)$ implies $\|f'\|_{L^\infty}<1$.
Then up to $\tau_R$, the first derivative $\|X^\infty(t,\omega)^{(1)}\|_{L^\infty}<1$
almost surely.
So up to $\tau_R$, $X^\infty(t,\omega)^{(1)}>-1$, or equivalently
$\widetilde{X}^\infty(t,\omega)^{(1)}>0$ almost surely.
Also by Proposition (\ref{intersection}),
$\widetilde{X}^\infty$ lives in the affine space $\widetilde{\diff}(S^1)$,
which means: every element $\widetilde{X}^\infty(t,\omega)$ is $C^\infty$.
Therefore, by Proposition (\ref{Difference}),
$\widetilde{X}^\infty$ lives in the group $\Diff(S^1)$ up to time $\tau_R$.

In the ``concatenating'' procedure (remark 3.13), the process
$\widetilde{Y}^\infty$ lives in the group $\Diff(S^1)$ up to time $\tau'_R$
for the same reason.
Because $\xi=\widetilde{X}^\infty(\tau_R)$, it is now a $\Diff(S^1)$-valued random variable.
So we have $\widetilde{Y}^\infty\circ\tilde{\xi}$ lives in $\Diff(S^1)$ up to time $\tau'_R$.
By concatenation, the process $\widetilde{Z}^\infty$ lives in $\Diff(S^1)$
up to time $\tau_R+\tau'_R$.
Because $\widetilde{X}^\infty$ is equivalent to $\widetilde{Z}^\infty$
up to time $\tau_R+\tau'_R$,
we have the process $\widetilde{X}^\infty$ lives in $\Diff(S^1)$
up to time $\tau_R+\tau'_R$.
We can carry out this ``concatenating'' procedure over and over again.
Therefore, the process $\widetilde{X}^\infty$ lives in $\Diff(S^1)$
up to the explosion time $\tau_\infty$ which is infinity by Proposition (\ref{part2}).

\end{proof}

Putting together Propositions (\ref{part1}), (\ref{part2}) and (\ref{part3}),
we have proved the main result of the paper:

\begin{theorem}\label{MainTheorem}
There is a unique $\widetilde{H}^k$-valued solution with continuous sample paths
to Equation \eqref{affineSDE} for all $k=0,1,2,\cdots$.
Furthermore, the solution is non-explosive and lives in the group $\Diff(S^1)$.
\end{theorem}

\bibliographystyle{amsplain}

\end{document}